\documentclass[11pt]{article}

\usepackage{amsmath}
\usepackage{latexsym}
\usepackage{amsfonts}
\usepackage{amssymb}
\usepackage{amscd}
\usepackage{theorem}
\newtheorem{theorem}{Theorem}[section]

\newtheorem{lemma}[theorem]{Lemma}

\newtheorem{corollary}[theorem]{Corollary}

{\theorembodyfont{\rmfamily}

}

\newenvironment{proof}{    
  \noindent
  \textbf{Proof.}}{
  \hfill $\Box$
  \vspace{3mm}
}

\numberwithin{equation}{section}

\newcommand{\N}{\mathbb{N}} 
\newcommand{\R}{\mathbb{R}} 
\newcommand{\C}{\mathbb{C}} 
\newcommand{\D}{\mathbb{D}} 

\begin{document}

\title{Every separable complex Fr\'echet space with a continuous norm is isomorphic to a space of holomorphic functions}

\author{Jos\'{e} Bonet  }

\date{}

\maketitle

\begin{abstract}
Extending a result of Mashreghi and Ransford, we prove that every complex separable infinite dimensional Fr\'echet space with a continuous norm is isomorphic to a space continuously included in a space of holomorphic functions on the unit disc or  the complex plane, which contains the polynomials as a dense subspace. As a consequence examples of nuclear Fr\'echet spaces of holomorphic functions without the bounded approximation exist.
\end{abstract}

\renewcommand{\thefootnote}{}
\footnotetext{\emph{2020 Mathematics Subject Classification.}
Primary: 46A04, secondary: 46A11, 46A32.}%
\footnotetext{\emph{Key words and phrases.} Spaces of holomorphic functions, Fr\'echet spaces, continuous norm, bounded approximation property}


\section{Introduction.}

Let $G$ be an open connected domain in the complex plane $\C$. We denote by $H(G)$ the Fr\'echet space of holomorphic functions on $G$, endowed with the topology of uniform convergence on compact subsets of $G$. A \textit{Fr\'echet space $E$ of holomorphic functions on the domain $G$} is a Fr\'echet space that is a subset of $H(G)$, such that the inclusion map $E \subset H(G)$ is continuous and $E$ contains the polynomials. By the closed graph theorem, the inclusion map $E \subset H(G)$ is continuous if and only if the point evaluations at the points of $G$ are continuous on $E$. If the polynomials are dense in $E$, then $E$ is separable. The set $\mathcal{P}$ of polynomials is dense in $H(G)$ if and only if $G$ is the whole complex plane or a simply connected domain; see \cite[Theorem 13.11]{Ru}. However, this assumption is not needed in our results below. Banach spaces of holomorphic functions on the unit disc $\D$ and on the complex plane $\C$ have been thoroughly investigated. We refer the reader for example to the books \cite{HKZ}, \cite{Z} and \cite{Z1}. H\"ormander Fr\'echet algebras of entire functions \cite{BG}, \cite{BLV}, \cite{MT}, the Fourier-Laplace transform of spaces of (ultra)-distributions \cite{BMT} and intersections of growth Banach spaces \cite{BLT}, \cite{HKZ} are natural examples of Fr\'echet spaces  of holomorphic functions. Vogt \cite{V13} proved that there are Fr\'echet spaces $E$ which are contained in $H(G)$ such that the inclusion $E \subset H(G)$ is not continuous. These examples are not Fr\'echet spaces of holomorphic functions on $G$ in our sense. Very recently, Mashreghi and Ransford \cite[Theorem 1.3]{MR} have shown that every separable, infinite-dimensional, complex Banach space $Y$ is isometrically isomorphic to a Banach space of holomorphic functions $X \subset H(\D)$ such that the polynomials are dense in $X$. The purpose of this note is to extend this result to the setting of Fr\'echet spaces and to derive a few consequences.

Our notation for functional analysis, in particular for Fr\'echet spaces and their duals, is standard. We refer the reader to \cite{J}, \cite{K}, \cite{MV} and \cite{PCB}. If $E$ is a Fr\'echet space, its topological dual is denoted by $E'$. The weak topology on $E$ is denoted by $\sigma(E,E')$ and the weak* topology on $E'$ by $\sigma(E',E)$. The linear span of a subset $A$ of $E$ is denoted by ${\rm span}(A)$. In what follows, we set $\N_0 := \N \cup \{ 0 \}$.

\section{Results.}

We start with the following observation. \textit{If a Fr\'echet space $F$ is isomorphic to a Fr\'echet space $E \subset H(G)$ of holomorphic functions on an open connected domain $G \subset \C$, then $F$ has a continuous norm.} Indeed, let $T: F \rightarrow E$ be a topological isomorphism and let $K \subset G$ be an infinite compact set. Since both $T$ and the inclusion $E \subset H(G)$ are continuous by assumption, there is a seminorm $p$ on $F$ such that $\sup_{z \in K} |T(x)(z)| \leq p(x)$ for each $x \in F$. If $p(x) = 0$ for some $x \in F$, then $T(x)(z)=0$ for each $z \in K$. Therefore, the holomorphic function $T(x)$ vanishes on $G$. Since $T$ is injective,  $x=0$. \\

We need the following Lemma to prove our main result.

\begin{lemma}\label{lemma_basic}
Let $F$ be a separable, infinite dimensional Fr\'echet space with a continuous norm. Then one can find two sequences $(e_n)_{n \in \N_0} \subset F$ and $(e'_n)_{n \in \N_0} \subset F'$ such that
\begin{itemize}
\item[(i)] ${\rm span}\{ e_n \ | \ n \in \N_0 \}$ is dense in $F$.
\item[(ii)] $(e'_n)_{n \in \N_0}$ is equicontinuous in $F'$.
\item[(iii)] ${\rm span}\{ e'_n \ | \ n \in \N_0 \}$ is dense in $(F',\sigma(F'F))$.
\item[(iv)] $\langle e_n, e'_m \rangle = 0$ if $n \neq m$ and $\langle e_n, e'_n \rangle = 1$ for each $n \in \N_0$. That is, $(e_n, e'_n)_{n \in \N_0}$ is a biorthogonal system.
\end{itemize}
\end{lemma}
\begin{proof}
We denote by $\tau$ the metrizable topology of the Fr\'echet space $F$ and by $p$ the continuous norm on $F$. Let $U:=\{ x \in F \ | \ p(x) \leq 1 \}$ be the unit ball of the norm $p$. The polar set $U^{\circ}$ of $U$ in $F'$ coincides with $\{ u \in F' \ | \ |u(x)| \leq p(x) \ {\rm for \ all} \ x \in F \}$ and, by Hahn-Banach theorem,  its linear span $H:= {\rm span}(U^{\circ}) = \cup_{s \in \N} s U^{\circ}$ is $\sigma(F',F)$-dense in $F'$.

Since $(F,\tau)$ is separable, there is a countable, infinite subset $A=(y_n)_{n \in \N_0}$ of $F$ which is $\tau$ dense in $F$. Hence, $A$ is also dense in the normed space $(F,p)$, and this normed space is separable, too. The topological dual of $(F,p)$ coincides with $H = {\rm span}(U^{\circ}) \subset F'$. Since $(F,p)$ is metrizable and separable, it follows from \cite[Corollary 2.5.13]{PCB} that $H$ is $\sigma(H,F)$-separable. We select a countable, infinite subset $A'=(v_n)_{n \in \N_0}$ of $H$ which is dense in $(H, \sigma(H,F))$. We now apply \cite[Proposition 2.3.2]{PCB} to find sequences $(x_n)_{n \in \N_0} \subset {\rm span}(A)$ and $(x'_n)_{n \in \N_0} \subset {\rm span}(A')$ such that $\langle x_n, x'_m \rangle = 0$ if $n \neq m$ and $\langle x_n, x'_n \rangle = 1$ for each $n \in \N_0$, and ${\rm span}(\{x_n, n \in \N_0\}) = {\rm span}(A)$ and ${\rm span}(\{x'_n, n \in \N_0\}) = {\rm span}(A')$.

Since for each $n \in \N_0$  we have $x'_n \in {\rm span}(A') \subset H = \cup_{s \in \N} s U^{\circ}$, then for each $n \in \N_0$ we can select $m(n) > 0$ such that $|x'_n(x)| \leq m(n) p(x)$ for each $x \in F$. We set $e'_n:= m(n)^{-1} x'_n$ and $e_n := m(n) x_n$ for each $n \in \N_0$ and we check that these sequences satisfy properties $(i)-(iv)$.

(i) ${\rm span}(\{e_n, n \in \N_0\}) = {\rm span}(\{x_n, n \in \N_0\}) = {\rm span}(A)$ is dense in $(F,\tau)$.

(ii) $|e'_n(x)| \leq p(x)$ for each $x \in F$ and each $n \in \N_0$. Therefore, $(e'_n)_{n \in \N_0}$ is equicontinuous in $F'$.

(iii) ${\rm span}(\{e'_n, n \in \N_0\}) = {\rm span}(\{x'_n, n \in \N_0\}) = {\rm span}(A')$ is $\sigma(H,F)$-dense in $H$. Since $H$ is $\sigma(F',F)$-dense in $F$, we conclude that ${\rm span}\{e'_n, n \in \N_0\}$ is $\sigma(F',F)$-dense in $F'$.

(iv) This follows easily from the definitions, since $(x_n, x'_n)_{n \in \N_0}$ is a biorthogonal system.
\end{proof}

Lemma \ref{lemma_basic} is a version for separable Fr\'echet spaces with a continuous norm of \cite[Lemma 2]{BP}, which was relevant in linear dynamics. See also \cite[Lemma 2.11]{BaMa}. Observe that the existence of a sequence $(e'_n)_{n \in \N_0} \subset F'$ satisfying (ii) and (iii) in Lemma \ref{lemma_basic} implies that the space $F$ has a continuous norm. In fact, by (ii) there is a seminorm $p$ on $F$ such that $\sup_{n \in \N_0}|e'_n(x)| \leq p(x)$ for all $x \in F$. Suppose that $p(y)=0$ for some $y \in F$. Then $e'_n(y)=0$ for each $n \in \N_0$. Condition (iii) implies that $y=0$.

The statement and proof of our next result are inspired by Mashreghi and Ransford \cite[Theorem 1.3]{MR}.

\begin{theorem}\label{maintheorem}
Let $F$ be a separable, infinite dimensional, complex Fr\'echet space with a continuous norm. Let $G$ be an open connected domain in $\C$. Then there exists a Fr\'echet space $E \subset H(G)$ of holomorphic functions on $G$ such that $F$ is isomorphic to $E$ and the polynomials are dense in $E$.
\end{theorem}
\begin{proof}
We apply Lemma \ref{lemma_basic} to the space $F$ to select the sequences $(e_n)_{n \in \N_0} \subset F$ and $(e'_n)_{n \in \N_0} \subset F'$ satisfying conditions $(i)-(iv)$. In particular, we find a continuous norm $p$ on $F$ such that $|e'_n(x)| \leq p(x)$ for each $x \in F$ and each $n \in \N_0$. Now take a sequence $(\alpha_n)_{n \in \N_0}$ of positive numbers such that $\sum_{n=0}^{\infty} \alpha_n k^n < \infty$ for each $k \geq 1$. Define $T:F \rightarrow H(G)$ by $T(x)(z):= \sum_{n=0}^{\infty} \alpha_n e'_n(x) z^n$ for each $x \in F$ and each $z \in G$. The operator $T$ is well-defined, linear and continuous. First of all, the series defining $T(x)$ converges and defines an entire function for each $x \in F$, since
$$\sum_{n=0}^{\infty} \alpha_n |e'_n(x)| |z|^n \leq p(x) \sum_{n=0}^{\infty} \alpha_n k^n < \infty$$
for each $|z| \leq k$ and each $k \geq 1$. The continuity of $T$ can be seen as follows: given an arbitrary compact subset $L$ of $G$, there is $k \geq 1$ such that $|z| \leq k$ for each $z \in L$. Then, for each $x \in F$ we get
$$
\sup_{z \in L} |T(x)(z)| = \sup_{z \in L} \big|\sum_{n=0}^{\infty} \alpha_n e'_n(x) z^n \big| \leq \big(\sum_{n=0}^{\infty} \alpha_n k^n \big) p(x).
$$
Moreover, the map $T$ is injective. Indeed, if $T(x)=0$ in $H(G)$, then $e'_n(x)=0$ for each $n \in \N_0$. Since
${\rm span}(\{ e'_n \ | \ n \in \N_0 \})$ is dense in $(F',\sigma(F',F))$, this implies that $x = 0$.

The monomials are clearly contained in $T(F)$, because $T((\alpha_n)^{-1} e_n)=z^n$ for each $n \in \N_0$. Therefore, each polynomial $P(z)= \sum_{n=0}^s a_n z^n = T(\sum_{n=0}^s a_n (\alpha_n)^{-1} e_n)$ belongs also to $T(F)$.

The linear map $T:F \rightarrow T(F) \subset H(G)$ is a continuous bijection. We endow $E:=T(F)$ with the metrizable, complete topology such that $T:F \rightarrow E$ is an isomorphism. Then $E$ is a separable Fr\'echet space of holomorphic functions on the domain $G$. The polynomials are contained in $E$ and they are also dense. This can be seen in the following way. By (i) in Lemma \ref{lemma_basic}, ${\rm span}(\{(\alpha_n)^{-1} e_n, n \in \N_0\}) = {\rm span}(\{ e_n \ | \ n \in \N_0 \})$ is dense in $F$. Then the set $\mathcal{P}$ of all polynomials satisfies $$\overline{\mathcal{P}} = \overline{T({\rm span}(\{(\alpha_n)^{-1} e_n), n \in \N_0\}))}=T(\overline{{\rm span}(\{(\alpha_n)^{-1} e_n, n \in \N_0\})}) = T(F) = E. $$
\end{proof}

A Fr\'echet space $F$ has the \textit{approximation property} if there is a net $(T_{\alpha})_{\alpha}$ of finite rank operators on $F$ such that $T_{\alpha} x$ converges to $x$ for each $x \in F$ uniformly on the compact subsets of $F$. A Fr\'echet space $F$ has the \textit{bounded approximation property} if there is an equicontinuous  net $(T_{\alpha})_{\alpha}$ of finite rank operators on $F$ such that $T_{\alpha} x$ converges to $x$ for each $x \in F$. If $F$ has the bounded approximation property, then it has the approximation property, since pointwise convergence and uniform convergence on compact sets coincide on the equicontinuous subsets of the space $L(F)$ of all continuous linear operators from $F$ into itself; see e.g.\ \cite[Theorem 39.4(2)]{K}.
Assume now that $F$ is also separable. In this case, every equicontinuous subset of the space $L(F)$ is metrizable for the topology of pointwise convergence by \cite[Theorem 39.5(9)]{K}. Therefore, one can apply the uniform boundedness principle to conclude that  a separable Fr\'echet space $F$ has the bounded approximation property if and only if there is a sequence $(T_n)_n$ of finite rank operators such that $\lim_{n \rightarrow \infty} T_n x = x$ for each $x \in F$. Every Fr\'echet space with a Schauder basis has the bounded approximation property and every nuclear Fr\'echet space has the approximation property. The problem of  Grothendieck whether every nuclear Fr\'echet space has the bounded approximation property was open for quite a while. The first counterexample was due to Dubinsky, and simpler examples were obtained by Vogt. We refer the reader to the introduction of Vogt's paper \cite{V10} for more information. In that paper Vogt presented an easy and transparent example of a nuclear Fr\'echet space failing the bounded approximation property and consisting of $C^{\infty}$-functions on a subset of $\R^3$. The first examples of nuclear Fr\'echet spaces with the bounded approximation property without basis are due to Mitiagin and Zobin; see pages 514 and ff. in \cite{J}. An easy example of a nuclear Fr\'echet space which consists of $C^{\infty}$-functions and has no Schauder basis was also given by Vogt in \cite{V05}. Theorem \ref{maintheorem} permits us to obtain in an abstract way examples among nuclear Fr\'echet spaces of holomorphic functions.

\begin{corollary}\label{bap_fn}
There are nuclear Fr\'echet spaces of holomorphic functions on the unit disc or the complex plane without the bounded approximation property, and there are others with the bounded approximation property without Schauder basis.
\end{corollary}

It is also a consequence of Theorem \ref{maintheorem} that classical counterexamples in the theory of Fr\'echet spaces exist in the frame of separable Fr\'echet spaces of holomorphic functions on an open connected domain. For example there are Fr\'echet Montel spaces which are not Schwartz, Fr\'echet Schwartz spaces without approximation property, non-distinguished Fr\'echet spaces and Fr\'echet spaces with a continuous norm such that their bidual is isomorphic to a countable product of Banach spaces, among many others. See more information about these examples  in \cite{BB} and \cite{MV}.

\vspace{.3cm}

\textbf{Acknowledgement.} This research was partially supported by the projects  MTM2016-76647-P and GV Prometeo/2017/102.



\noindent \textbf{Author's address:}%
\vspace{\baselineskip}%

Jos\'e Bonet: Instituto Universitario de Matem\'{a}tica Pura y Aplicada IUMPA,
Universitat Polit\`{e}cnica de Val\`{e}ncia,  E-46071 Valencia, Spain

email: jbonet@mat.upv.es \\

\end{document}